\documentclass[12pt]{article}

\usepackage{mathrsfs}
\usepackage{amsmath,extpfeil}
\usepackage{stmaryrd}
\usepackage{mathrsfs}
\usepackage{url}
\usepackage{indentfirst}
\usepackage{enumerate}
\usepackage{amsmath,amsfonts,amssymb,amsthm}
\usepackage{amsmath,amssymb,amsthm,amscd}
\usepackage{graphicx,mathrsfs}
\usepackage{appendix}
\usepackage[numbers,sort&compress]{natbib}
\usepackage{color}
\usepackage[colorlinks, linkcolor=blue, citecolor=blue]{hyperref}

\usepackage{sidecap}
\usepackage{float}
\usepackage{extarrows}
\usepackage{booktabs}
\usepackage{verbatim}
\usepackage[usenames,dvipsnames]{xcolor}

\newtheorem{theorem}{Theorem}[section]
\newtheorem{lemma}[theorem]{Lemma}

\theoremstyle{definition}
\newtheorem{remark}[theorem]{Remark}

\def\XXint#1#2#3{{\setbox0=\hbox{$#1{#2#3}{\int}$}
         \vcenter{\hbox{$#2#3$}}\kern-.5\wd0}}

\numberwithin{equation}{section}

\allowdisplaybreaks[4]

\begin{document}

\title{On traveling waves and global existence for a  nonlinear  Schr\"odinger system with three waves interaction }
\author{ Yuan  Li}
\date{}
\maketitle

\begin{abstract}
In this paper, we consider  three components system of  nonlinear Schr\"odinger equations related to the Raman amplification in a plasma.  By using variational method, a new result on the existence of traveling wave solutions are obtained  under the non-mass resonance condition. We also study the new global existence result for oscillating data. Both of our results essentially due to the absence of Galilean symmetry in the system.


\end{abstract}

\section{Introduction and Main Results}
\noindent
In this paper, we consider the traveling wave solution and global existence for a Schr\"odinger system
\begin{equation}\label{equ:system}
    \begin{cases}
     i\gamma_1\partial_tu+\Delta u+w\bar{v}=0,\\
     i\gamma_2\partial_tv+\Delta v+w\bar{u}=0,\\
    i\gamma_3\partial_tw+\Delta w+uv=0,
    \end{cases}
\end{equation}
where $\gamma_1,\gamma_2,\gamma_3$ are positive real constants, $\bar{u}$ is the complex conjugate of $u$.

System \eqref{equ:system} arises from the nonlinear interaction between a laser beam and a plasma in the Raman process which is a nonlinear instability phenomenon. This phenomenon is composed by three Schr\"odinger equations coupled to a wave equation and reads in a suitable dimensionless form:
\begin{equation}\label{equ:o}
    \begin{cases}
    (i(\partial_t+v_C\partial_y)+\alpha_1\partial_y^2+\alpha_2\Delta)A_C=\frac{b^2}{2}nA_C-\gamma(\nabla\cdot E)A_Re^{-i\theta},\\
    (i(\partial_t+v_R\partial_y)+\beta_1\partial_y^2+\beta_2\Delta)A_R=\frac{bc}{2}nA_R-\gamma(\nabla\cdot E^*)A_Ce^{-i\theta},\\
    (i\partial_t+\delta_1\Delta)E=\frac{b}{2}nE+\gamma\nabla(A_R^*A_Ce^{i\theta}),\\
    (\partial_t^2-v_s^2\Delta)n=a\delta(|E|^2+b|A_C|^2+c|A_R|^2),
    \end{cases}
\end{equation}
where $A_C, A_R$ and $E$ are complex vectors and are respectively the incident laser field, the back-scattered Raman field and electronic plasma, $n$ is the variation of the density of the ions and $\theta=(k_1y-\omega_2t)$ where $\omega_1=k_1^2\delta$. For more details of definition and the physical explanation of all variables and notations can be found in \cite{CC2004DIE,CCO2009poincare,BDST2002optical,CDS2016nonlinearity}. On a similar method to \cite{CCO2009poincare}, let $E=Fe^{i\theta}$, the system \eqref{equ:o} reads
\begin{equation*}
\begin{cases}
    (iv_C\partial_y+\alpha_2\Delta)A_C=-\gamma ik_1FA_R,\\
    (iv_R\partial_y+\beta_2\Delta)A_R=\gamma ik_1F^{*}A_C,\\
    (2ik_1\partial_y+\delta_1\Delta)F=\gamma ik_1 A_R^*A_C.
    \end{cases}
\end{equation*}
Hence, system \eqref{equ:system} can be regard as the special case of above.

On the other hand, system \eqref{equ:system}  is also related to the second harmonic generation (SHG) in a bulk medium. It can be written as
\begin{equation*}
    \begin{cases}
    2ik_1\partial_zu_1+\Delta u_1+\gamma_1u_3\bar{u}_2e^{-i\Delta kz}=0,\\
    2ik_2\partial_zu_2+\Delta u_2+\gamma_2u_3\bar{u}_1e^{-i\Delta kz}=0,\\
    2ik_3\partial_zu_3+\Delta u_3+\gamma_3u_1u_2e^{i\Delta kz}=0.
    \end{cases}
\end{equation*}
For more details about the physical meaning of this system, one can refer to  \cite{CDS2016nonlinearity,BDST2002optical} and the  references therein.

Let us review some basic properties of system \eqref{equ:system}. The Cauchy problem \eqref{equ:system} is an infinite-dimensional Hamiltonian system, which has the following three conservation laws:

Mass:
\begin{align}\label{mass}
\begin{cases}
M(u,v,w)=\gamma_1\|u\|_{L^2}^2+\gamma_2\|v\|_{L^2}^2+2\gamma_3\|w\|_{L^2}^2,\\
    M_1(u,w)=\gamma_1\|u\|_{L^2}^2+\gamma_3\|w\|_{L^2}^2,\\
    M_2(v,w)=\gamma_2\|v\|_{L^2}^2+\gamma_3\|w\|_{L^2}^2,\\
    M_3(u,v)=\gamma_1\|u\|_{L^2}^2-\gamma_2\|v\|_{L^2}^2.
\end{cases}
\end{align}

Energy:
\begin{align}\label{energy}
    E(u,v,w)=\frac{1}{2}K(u,v,w)-\Re\int uv\bar{w},
\end{align}
where
\begin{align}\label{def:K}
    K(u,v,w)=\int|\nabla u|^2+\int|\nabla v|^2+\int|\nabla w|^2.
\end{align}

Momentum:
\begin{align}\label{momentum}
    P(u,v,w)=\gamma_1(i\nabla u,u)+\gamma_2(i\nabla v,v)+\gamma_3(i\nabla w,w),
\end{align}
where
\[(u,v)=\Re \int u\bar{v}.\]

The system \eqref{equ:system} also has the following symmetries:

Gauge transform:
\[(u,v,w)\mapsto \left(e^{i\theta}u,e^{i\theta}v,e^{2i\theta}w\right), ~~\theta\in\mathbb{R},\]

Scaling:
\[(u,v,w)\mapsto \left(\lambda^2u(\lambda^2t,\lambda x),\lambda^2v(\lambda^2t,\lambda x),\lambda^2w(\lambda^2t,\lambda x)\right),~~\lambda>0.\]
We note that the spatial Sobolev $\dot{H}^{s_c}$ norm with $s_c=\frac{N}{2}-2$ is invariant under the scaling. In particular, \eqref{equ:system} is $L^2$ critical if $s_c=0$ ($N=4$), and energy critical if $s_c=1$ ($N=6$). When $\gamma_1+\gamma_2=\gamma_3$, which is called the mass resonance condition, then system \eqref{equ:system} is invariant under the Galilean transformation
\begin{align*}
    (u,v,w)\mapsto\left(e^{\frac{i\gamma_1}{2}c\cdot x-\frac{i\gamma_1}{4}|c|^2t}u(t,x-ct),e^{\frac{i\gamma_2}{2}c\cdot x-\frac{i\gamma_2}{4}|c|^2t}v(t,x-ct),e^{i\frac{\gamma_3}{2}c\cdot x-\frac{i\gamma_3}{4}|c|^2t}w(t,x-ct)\right).
\end{align*}
for any $c\in\mathbb{R}^N$.

From \cite{NP2021CCM,NP2022CVPDE}, it is well known that the Cauchy problem local well-posedness in $H^1(\mathbb{R}^N)\times H^1(\mathbb{R}^N)\times H^1(\mathbb{R}^N)$, $1\leq N\leq 6$. In particular, if $N=4$, a classical criterion of global-in-time existence for $H^1(\mathbb{R}^4)\times H^1(\mathbb{R}^4)\times H^1(\mathbb{R}^4)$ initial data is derived by
using the sharp Gagliardo-Nirenberg inequality (see Appendix \ref{section:GN})
\begin{align}\label{GN}
    \int uvw\leq \frac{1}{2}\frac{M(u,v,w)^{\frac{1}{2}}}{M(\phi,\psi,\varphi)^{\frac{1}{2}}} K(u,v,w),
\end{align}
for $(u,v,w)\in H^1\times H^1\times H^1$, where $(\phi,\psi,\varphi)$ is a standing wave solution of \eqref{equ:system}, the existence of standing wave solutions can be found in Appendix \ref{section:ground} or the general quadratic-type nonlinearities \cite[Section 4]{NP2021CCM}. It follows from \eqref{GN} and the conservation laws of mass and energy (see \eqref{mass},\eqref{energy}) that if the initial data $(u_0,v_0,w_0)\in H^1(\mathbb{R}^4)\times H^1(\mathbb{R}^4)\times H^1(\mathbb{R}^4) $ satisfies
\begin{align}\label{mass:con}
    M(u_0,v_0,w_0)<M(\phi,\psi,\varphi),
\end{align}
then the corresponding $H^1$-solution of \eqref{equ:system} is global and bounded. The condition \eqref{GN} is sharp in general. Indeed, if the initial data $(u_0,v_0,w_0)\in H^1(\mathbb{R}^4)\times H^1(\mathbb{R}^4)\times H^1(\mathbb{R}^4) $ satisfies
\begin{align*}
     M(u_0,v_0,w_0)>M(\phi,\psi,\varphi),~~E(u_0,v_0,w_0)<0.
\end{align*}
Then the corresponding $H^1$-solution blows up in infinite time, see \cite[Theorem 5.11]{NP2021CCM}.

If we set $\gamma_1=\gamma_2$ and $u=v$, then \eqref{equ:system} reduces to
\begin{equation}\label{equ:2wave}
    \begin{cases}
    i\partial_t+\Delta u+2\bar{u}w=0,\\
    i\partial_t w+k\Delta w+ u^2=0,
    \end{cases}
\end{equation}
which is a simplified model describing Raman amplification phenomena in a plasma, see \cite{HOT2013Poincare}. Hayashi, Ozawa and Tanaka \cite{HOT2013Poincare} studied the Cauchy problem \eqref{equ:2wave} in $H^1$, $L^2$ and the weighted $L^2$ space under mass resonance condition; they also obtained the ground state by using the variational method. Recently, Fukaya, Hayashi and Inui \cite{FHI2022MA} considered the traveling wave solution and global existence of \eqref{equ:2wave} under the non-mass resonance condition. Noguera and Pastor \cite{NP2021CCM,NP2022CVPDE} studied the blowup solution  of the general Schr\"odinger system with quadratic interaction. For more blowup results of the Schr\"odinger system with quadratic interaction, one can see \cite{DF2021ZAMP} and the references therein.

Inspired by \cite{FHI2022MA}, in this paper, we aim to study traveling wave solution and give the new global existence result for oscillating initial data of system \eqref{equ:system}.

\subsection{Traveling wave solutions}
Now, we consider traveling waves solutions of \eqref{equ:system} in the form of
\begin{align}\label{traveling}
    \left(u_{c,w}(t,x),v_{c,w}(t,x),w_{c,w}(t,x)\right)=\left(e^{i\omega t}\phi(x-ct),e^{i\omega t}\psi(x-ct),e^{i2\omega t}\varphi(x-ct)\right).
\end{align}
In particular, when $c=0$, \eqref{traveling} is the form of standing wave solution, we give the existence of this solution under the condition $\gamma_1+\gamma_2=\gamma_3$ in the Appendix \ref{section:ground} or the general result can be found in \cite{NP2021CCM}.  When $\gamma_1+\gamma_2=\gamma_3$, one can obtain traveling wave solutions from the standing wave solutions through the Galilean transformation. On the other hand, when $\gamma_1+\gamma_2\neq\gamma_3$, such a construction does not work due to the lack of Galilean symmetry, Therefore, it is meaningful to study the existence of traveling wave solutions of \eqref{equ:system}.

From now on, we assume that
\[\gamma_1+\gamma_2\neq\gamma_3.\]
For $(\omega,c)\in\mathbb{R}\times\mathbb{R}^N$, \eqref{traveling} is a solution of \eqref{equ:system} if and only if $\left(\phi(x-ct),\psi(x-ct),\varphi(x-ct)\right)$ is a solution of the stationary system
\begin{equation}\label{equ:s}
    \begin{cases}
    -\Delta\phi+\gamma_1\omega\phi+i\gamma_1c\cdot\nabla\phi-\varphi\bar{\psi}=0,\\
    -\Delta\psi+\gamma_2\omega\psi+i\gamma_2 c\cdot\nabla\psi-\varphi\bar{\phi}=0,\\
    -\Delta\varphi+2\gamma_3\omega\varphi+i\gamma_3c\cdot\nabla\varphi-\phi\psi=0.
    \end{cases}
\end{equation}
The energy functional with respect to \eqref{equ:s} is defined by
\begin{align}\label{def:functional}
    S_{\omega,c}(u,v,w)=&E(u,v,w)+\frac{1}{2}\omega M(u,v,w)+\frac{1}{2}c\cdot P(u,v,w),
\end{align}
where $E$, $M$ and $P$ are given by \eqref{energy}, \eqref{mass} and \eqref{momentum}, respectively.
For the sake of argument, we rewrite the energy functional \eqref{def:functional} as
\begin{align}\label{def:f2}
    S_{\omega,c}(u,v,w)=&\frac{1}{2}\int\left|\nabla\left(e^{-\frac{i\gamma_1}{2}c\cdot x}u\right)\right|^2+\frac{1}{2}\left(\gamma_1\omega-\frac{\gamma_1^2|c|^2}{4}\right)\|u\|_{L^2}^2\notag\\
    &+\frac{1}{2}\int\left|\nabla\left(e^{-\frac{i\gamma_2}{2}c\cdot x}v\right)\right|^2+\frac{1}{2}\left(\gamma_2\omega-\frac{\gamma_2^2|c|^2}{4}\right)\|v\|_{L^2}^2\notag\\
    &+\frac{1}{2}\int\left|\nabla\left(e^{-\frac{i\gamma_3}{2}c\cdot x}w\right)\right|^2+\frac{1}{2}\left(2\gamma_3\omega-\frac{\gamma_3^2|c|^2}{4}\right)\|w\|_{L^2}^2-\Re\int uv\bar{w}.
\end{align}
Without loss of generality, we assume that
\[0<\gamma_1\leq \gamma_2.\]
We define the function space $X_{\omega,c}$ by
\begin{align*}
    X_{\omega,c}=\left\{(u,v,w): \left(e^{-\frac{i\gamma_1}{2}c\cdot x}u,e^{-\frac{i\gamma_2}{2}c\cdot x}v,e^{-\frac{i\gamma_3}{2}c\cdot x}w\right)\in \Tilde{X}_{\omega,c} \right\},
\end{align*}
where
\begin{align*}
   \Tilde{X}_{\omega,c}= \begin{cases}
    &H^1(\mathbb{R}^N)\times H^1(\mathbb{R}^N)\times H^1(\mathbb{R}^N),~~\omega>\max\left\{\frac{\gamma_1|c|^2}{4},\frac{\gamma_2|c|^2}{4},\frac{\gamma_3|c|^2}{8}\right\},\\
    &H^1(\mathbb{R}^N)\times H^1(\mathbb{R}^N)\times \dot{H}^1(\mathbb{R}^N),~~\omega=\frac{\gamma_3|c|^2}{8},~~\gamma_1\leq\gamma_2<\frac{\gamma_3}{2},\\
    &H^1(\mathbb{R}^N)\times \dot{H}^1(\mathbb{R}^N)\times \dot{H}^1(\mathbb{R}^N),~~\omega=\frac{\gamma_3|c|^2}{8},~~\gamma_1<\gamma_2=\frac{\gamma_3}{2},\\
    &H^1(\mathbb{R}^N)\times\dot{H}^1(\mathbb{R}^N)\times H^1(\mathbb{R}^N),~~\omega=\frac{\gamma_2|c|^2}{4},~~\max\left\{\gamma_1,\frac{\gamma_3}{2}\right\}<\gamma_2,\\
    &\dot{H}^1(\mathbb{R}^N)\times\dot{H}^1(\mathbb{R}^N)\times H^1(\mathbb{R}^N),~~\omega=\frac{\gamma_2|c|^2}{4},~~\frac{\gamma_3}{2}< \gamma_1=\gamma_2.
    \end{cases}
\end{align*}
We note that $S_{\omega,c}$ is defined on $X_{\omega,c}$ if
\begin{align}\label{D:condition}
    \begin{cases}
    (A)~~1\leq N\leq 5,~~\omega>\max\left\{\frac{\gamma_1|c|^2}{4},\frac{\gamma_2|c|^2}{4},\frac{\gamma_3|c|^2}{8}\right\},\\
    (B)~~3\leq N\leq 5,~~\omega=\frac{\gamma_3|c|^2}{8},~~\gamma_1\leq\gamma_2<\frac{\gamma_3}{2},\\
    (C)~~4\leq N\leq 5,~~\omega=\frac{\gamma_3|c|^2}{8},~~\gamma_1<\gamma_2=\frac{\gamma_3}{2},\\
    (D)~~3\leq N\leq 5,~~\omega=\frac{\gamma_2|c|^2}{4},~~\max\left\{\gamma_1,\frac{\gamma_3}{2}\right\}<\gamma_2,\\
    (E)~~4\leq N\leq 5,~~\omega=\frac{\gamma_2|c|^2}{4},~~\frac{\gamma_3}{2}< \gamma_1=\gamma_2.
    \end{cases}
\end{align}
The dimensional conditions in $(B)$, $(C)$, $(D)$ and $(E)$ come from the Sobolev embedding $\dot{H}^1\hookrightarrow L^{\frac{2N}{N-2}}$, which is used to control the nonlinear term.

We denote the set of all nontrivial solutions of \eqref{equ:s} by
\begin{align*}
    \mathcal{A}_{\omega,c}=\{(\phi,\psi,\varphi)\in X_{\omega,c}~:~(\phi,\psi,\varphi)\neq(0,0,0),~S^{\prime}_{\omega,c}(\phi,\psi,\varphi)=0\}
\end{align*}
and the set of all boosted ground states
\begin{align*}
    \mathcal{G}_{\omega,c}=\{(\phi,\psi,\varphi)\in\mathcal{A}_{\omega,c},~S_{\omega,c}(\phi,\psi,\varphi)\leq S_{\omega,c}(\phi_1,\psi_1,\varphi_1),~\text{for all}~ (\phi_1,\psi_1,\varphi_1)\in\mathcal{A}_{\omega,c}\}.
\end{align*}
In particular, if $c=0$,  $\mathcal{G}_{\omega,0}$ is the set of all ground states.

For $(\phi,\psi,\varphi)$ satisfying \eqref{equ:s},  let
\begin{align*}
    (\phi,\psi,\varphi)=\left(e^{\frac{i\gamma_1}{2}c\cdot x}\Tilde\phi,e^{\frac{i\gamma_2}{2}c\cdot x}\Tilde\psi,e^{i\frac{\gamma_3}{2}c\cdot x}\Tilde\varphi\right).
\end{align*}
Then $(\Tilde\phi,\Tilde\psi,\Tilde\varphi)$ satisfies
\begin{align}\label{sys:complex}
    \begin{cases}
    -\Delta\Tilde{\phi}+\left(\omega-\frac{|c|^2}{4}\right)\Tilde{\phi}-e^{i\left(\frac{\gamma_3}{2}-\frac{\gamma_1+\gamma_2}{2}\right)c\cdot x}\Tilde{\varphi}\bar{\Tilde{\psi}} =0,\\
    -\Delta\Tilde{\psi}+\left(\omega-\frac{|c|^2}{4}\right)\Tilde{\psi}-e^{i\left(\frac{\gamma_3}{2}-\frac{\gamma_1+\gamma_2}{2}\right)c\cdot x}\Tilde{\varphi}\bar{\Tilde{\phi}} =0,\\
    -\Delta\Tilde{\varphi}+\left(2\gamma_3\omega-\frac{\gamma_3^2|c|^2}{4}\right)\Tilde{\varphi}-e^{i\left(\frac{\gamma_1+\gamma_2}{2}-\frac{\gamma_3}{2}\right)c\cdot x}\Tilde{\phi}\Tilde{\psi}=0.\\
    \end{cases}
\end{align}
Note that the solutions of \eqref{sys:complex} with $\gamma_1+\gamma_2\neq\gamma_3$ (non mass resonance condition) are essentially non-radial and complex-valued.

Now we state our first result.
\begin{theorem}\label{Thm1}
Assume \eqref{D:condition} and assume further $N=5$ for the cases $(C)$ and $(E)$. Then there exist traveling wave solutions of \eqref{equ:system} in the form of \eqref{traveling}.
\end{theorem}

\begin{remark}
 In the cases $(C)$ and $(E)$ with $N=4$, we cannot rule out the vanishing of minimizing sequences. This difficulty is related to sequence compactness
\end{remark}

Next, we show that nonexistence of nontrivial solutions for \eqref{sys:complex} with $\gamma_1+\gamma_2=\gamma_3$ (mass resonance condition).
\begin{theorem}\label{Thm2}
Let $1\leq N\leq 5$,  $\gamma_1+\gamma_2=\gamma_3$ and $\omega=\frac{\gamma_1|c|^2}{4}=\frac{\gamma_2|c|^2}{4}$. If $(\phi,\psi,\varphi)$ is a solution of \eqref{equ:s} satisfying $\left(e^{-\frac{i\gamma_1}{2}c\cdot x}\phi,e^{-\frac{i\gamma_2}{2}c\cdot x}\psi,e^{-\frac{i\gamma_3}{2}c\cdot x}\varphi\right)\in \dot{H}^1(\mathbb{R}^N)\times\dot{H}^1(\mathbb{R}^N)\times\dot{H}^1(\mathbb{R}^N)$ and $\phi\psi\bar\varphi\in L^1$, then $(\phi,\psi,\varphi)=(0,0,0)$.
\end{theorem}
{\bf Comment:}

 For the cases $(B)$, $(C)$, $(D)$ and $(E)$, the coefficient of either $\phi$ and $\psi$ or $\varphi$ in \eqref{sys:complex} vanishes, which corresponds to ``zero mass" case in the elliptic problem. Zero mass problem appear in various situations, see \cite{DD2002JMPA,LN2020CVPDE,MP1990ARMA,MM2014JEMS,T1976AMPA,HLZ2021DCDS}. Our existence result for the cases $(B)$, $(C)$, $(D)$ and $(E)$ come from the fact that the coefficient of $L^2$-norms in the energy functional \eqref{def:f2} do not vanish at the same time. This not true for the case $\gamma_1+\gamma_2=\gamma_3$, see Theorem \ref{Thm2}. Therefore, the lack of symmetries yields the new and nontrivial existence result, and the solutions of \eqref{sys:complex} are non-radial and complex-valued.

\subsection{Global existence}
Finally, we give the following global result for arbitrarily large data with modification of oscilations in four dimension.
\begin{theorem}\label{Thm3}
Assume $N=4$ and $\gamma_1+ \gamma_2\neq\gamma_3$.  Let $c\in\mathbb{R}^4$,  $u_0,v_0,w_0\in H^1(\mathbb{R}^4)$ and
\begin{align}\label{initial}
    (u_{0,c},v_{0,c},w_{0,c})=\left(e^{\frac{i\gamma_1}{2}c\cdot x}u_0,e^{\frac{i\gamma_2}{2}c\cdot x}v_0,e^{\frac{i\gamma_3}{2}c\cdot x}w_0\right).
\end{align}
Then the following statements hold.

(i) If $\gamma_3>\gamma_1+\gamma_2$,  there exist  $A_0,A_1>0$ such that $\max\left\{\|u_0\|_{L^2}^2,\|v_0\|_{L^2}^2\right\}<A_0$  and $|c|\geq A_1$, $c\in\mathbb{R}^4$, then  the $H^1$-solution of \eqref{equ:system} with the initial data \eqref{initial} exists globally in time.

(ii) If $\gamma_3<\gamma_1+\gamma_2$ and $\gamma_1< \gamma_2$, there exists $B_0,B_1>0$ such that $\max\left\{\|u_0\|_{L^2}^2,\|w_0\|_{L^2}^2\right\}<B_0$ and $|c|>B_1$, then the $H^1$-solution of \eqref{equ:system} with the initial data \eqref{initial} exists globally in time.

(ii) If $\gamma_3<\gamma_1+\gamma_2$ and $\gamma_1> \gamma_2$, there exists $C_0,C_1>0$ such that $\max\left\{\|v_0\|_{L^2}^2,\|w_0\|_{L^2}^2\right\}<C_0$ and $|c|>C_1$, then the $H^1$-solution of \eqref{equ:system} with the initial data \eqref{initial} exists globally in time.

In particular, if  $\gamma_3<\gamma_1+\gamma_2$ and $\gamma_1= \gamma_2$, there  exists $D_0,D_1>0$ such that $\|w_0\|_{L^2}^2<D_0$ and $|c|>D_1$, then the $H^1$-solution is global.
\end{theorem}

{\bf Comments:}

1. When $\gamma_1+\gamma_2=\gamma_3$, Under the transformation of initial data
\begin{align}\label{initial:1}
    (u_{0,c},v_{0,c},w_{0,c})=\left(e^{\frac{i}{2}c\cdot x}u_0,e^{\frac{i}{2}c\cdot x}v_0,e^{\frac{i\gamma_3}{2}c\cdot x}w_0\right),
\end{align}
global properties of the solution do not change due to the Galilean invaraince.  However, when $\gamma_1+\gamma_2\neq\gamma_3$, the momentum change of the initial data by \eqref{initial:1}, this essentially influences global properties of the solution, which comes from the lack of Galilean invariance.

2. Theorem \ref{Thm1} and Theorem \ref{Thm3} implies that if we consider \eqref{equ:system} in the non-radial regime, there is an essential difference in global dynamics between the cases of mass resonance and non-mass resonance.

3. When $N=4$, in the radial case, the global existence  condition \eqref{mass:con} of the $H^1$-solution is optimal. However, Theorem \ref{Thm3} means that we only need to restrict the $L^2$-norm of  some initial data to obtain the global existence. This is due to the lack of Galilean invariance in the system.

This paper is organized as follows: in Section 2, we show the existence of traveling wave solution of system \eqref{equ:s} by the variational method; in Section 3, the non-existence of solution for \eqref{equ:s} will be derived; in Section 4, we establish the global existence result for oscillating initial data  and the finally section is Appendix.

\section{Existence of the traveling wave solution}
In this section, we aim to  prove the existence of traveling wave solutions by solving variational problems on the Nehari manifold.

For the sake of simple, we use the following notation. We set
\begin{align*}
    Q_{\omega,c}(u,v,w)=&\frac{1}{2}K(u,v,w)
    +\frac{\omega\gamma_1}{2}\|u\|_{L^2}^2+\frac{\omega\gamma_2}{2}\|v\|_{L^2}^2+\gamma_3\omega\|w\|_{L^2}^2+\frac{1}{2}c\cdot P(u,v,w),\\
    V(u,v,w)=&\Re\int uv\bar{w}.
\end{align*}
The energy functional \eqref{def:functional} reads
\[S_{\omega,c}(u,v,w)=Q_{\omega,c}(u,v,w)-V(u,v,w).\]
Now we introduce the Nehari functional
\begin{align*}
    N_{\omega,c}(u,v,w)=\partial_\lambda S_{\omega,c}(\lambda u,\lambda v,\lambda w)|_{\lambda=1}=2Q_{\omega,c}(u,v,w)-3V(u,v,w).
\end{align*}
We also set
\begin{align*}
    \widetilde{Q}_{\omega,c}(u,v,w)=&\frac{1}{2}K(u,v,w)
    +\frac{1}{2}\left(\gamma_1\omega-\frac{\gamma_1^2|c|^2}{4}\right)\|u\|_{L^2}^2\\
    &+\frac{1}{2}\left(\gamma_2\omega-\frac{\gamma_2^2|c|^2}{4}\right)\|v\|_{L^2}^2+\frac{1}{2}\left(2\gamma_3\omega-\frac{\gamma_3^2|c|^2}{4}\right)\|w\|_{L^2}^2,\\
    \widetilde{V}(u,v,w)=&\Re\int e^{i\left(\frac{\gamma_1+\gamma_2}{2}-\frac{\gamma_3}{2}\right)c\cdot x} uv\bar{w}.
\end{align*}
If
\[(\Tilde{u},\Tilde{v},\Tilde{w})=\left(e^{-\frac{i\gamma_1}{2}c\cdot x}u,e^{-\frac{i\gamma_2}{2}c\cdot x}v,e^{-i\frac{\gamma_3}{2}c\cdot x}w\right),\]
then we have the following relations
\begin{align*}
    \Tilde{Q}_{\omega,c}(\Tilde{u},\Tilde{v},\Tilde{w})=Q_{\omega,c}(u,v,w),~~\widetilde{V}(\Tilde{u},\Tilde{v},\Tilde{w})=V(u,v,w).
\end{align*}
The corresponding  functionals
\begin{align*}
    \Tilde{S}_{\omega,c}(u,v,w)=&\Tilde{Q}_{\omega,c}(u,v,w)-\Tilde{V}(u,v,w),\\
    \Tilde{N}_{\omega,c}=&\partial_\lambda\Tilde{S}_{\omega,c}(\lambda u,\lambda v,\lambda w)|_{\lambda=1}=2\Tilde{Q}_{\omega,c}(u,v,w)-3\Tilde{V}(u,v,w).
\end{align*}
The minimization problem
\begin{align*}
    \mu_{\omega,c}:=\inf\{S_{\omega,c}(\phi,\psi,\varphi)~:~(\phi,\psi,\varphi)\in \mathcal{N}_{\omega,c}\},
\end{align*}
where
\begin{align*}
    \mathcal{N}_{\omega,c}=\{(\phi,\psi,\varphi)\in X_{\omega,c},~(\phi,\psi,\varphi)\neq(0,0,0),~N_{\omega,c}(\phi,\psi,\varphi)=0\}.
\end{align*}
We define the minimizers $a_{\omega,c}$ by
\begin{align*}
    a_{\omega,c}=\{(\phi,\psi,\varphi)\in \mathcal{N}_{\omega,c}~:~S_{\omega,c}(\phi,\psi,\varphi)=\mu_{\omega,c}\}.
\end{align*}
We also use the following notations.
\begin{align*}
     &\widetilde{\mathcal{A}}_{\omega,c}=\{(\phi,\psi,\varphi)\in \Tilde{X}_{\omega,c}~:~(\phi,\psi,\varphi)\neq(0,0,0),~\Tilde{S}^{\prime}_{\omega,c}(\phi,\psi,\varphi)=0\},\\
      &\widetilde{\mathcal{G}}_{\omega,c}=\{(\phi,\psi,\varphi)\in\widetilde{\mathcal{A}}_{\omega,c},~\Tilde{S}_{\omega,c}(\phi,\psi,\varphi)\leq \Tilde{S}_{\omega,c}(\phi_1,\psi_1,\varphi_1),~\text{for all}~ (\phi_1,\psi_1,\varphi_1)\in\Tilde{\mathcal{A}}_{\omega,c}\},\\
      &\Tilde{\mu}_{\omega,c}:=\inf\{\Tilde{S}_{\omega,c}(\phi,\psi,\varphi)~:~(\phi,\psi,\varphi)\in \Tilde{\mathcal{N}}_{\omega,c}\},\\
      &\Tilde{\mathcal{N}}_{\omega,c}=\{(\phi,\psi,\varphi)\in \Tilde{X}_{\omega,c},~(\phi,\psi,\varphi)\neq(0,0,0),~\Tilde{N}_{\omega,c}(\phi,\psi,\varphi)=0\},\\
       &\Tilde{a}_{\omega,c}=\{(\phi,\psi,\varphi)\in \Tilde{\mathcal{N}}_{\omega,c}~:~\Tilde{S}_{\omega,c}(\phi,\psi,\varphi)=\Tilde{\mu}_{\omega,c}\}.
\end{align*}
Notice that
\begin{align*}
    (\phi,\psi,\varphi)\in {\mathcal{G}}_{\omega,c}\Leftrightarrow (\Tilde{\phi},\Tilde{\psi},\Tilde{\varphi})\in \Tilde{\mathcal{G}}_{\omega,c},\\
     (\phi,\psi,\varphi)\in a_{\omega,c}\Leftrightarrow (\Tilde{\phi},\Tilde{\psi},\Tilde{\varphi})\in \Tilde{a}_{\omega,c},
\end{align*}
and
\begin{align*}
    \mu_{\omega,c}=\Tilde{\mu}_{\omega,c}.
\end{align*}
 Now we state the main result in this section.
 \begin{theorem}\label{lemma:nonempty}
If \eqref{D:condition} holds and assume $N=5$ for the cases $(C)$ and $(E)$, then
 $$\mathcal{G}_{\omega,c}=a_{\omega,c}\neq \emptyset.$$
 In fact, this is equivalent to
 $\widetilde{\mathcal{G}}_{\omega,c}=\Tilde{a}_{\omega,c}\neq \emptyset.$
 \end{theorem}
 If this theorem is true, then we can easily obtain the Theorem \ref{Thm1}.  In order to prove the theorem \ref{lemma:nonempty}, we need the following lemmas.

 \begin{lemma}\label{lemma:a}
 If \eqref{D:condition} holds, then $\Tilde{a}_{\omega,c} \subset\widetilde{\mathcal{G}}_{\omega,c}$.
 \end{lemma}
 \begin{proof}
 Let $(\phi,\psi,\varphi)\in \Tilde{a}_{\omega,c}$. Since $\Tilde{N}_{\omega,c}(\phi,\psi,\varphi)=0$ and $(\phi,\psi,\varphi)\neq 0$, we have
 \begin{align}\label{exist:1}
     \left(\Tilde{N}^{\prime}_{\omega,c}(\phi,\psi,\varphi),(\phi,\psi,\varphi)\right)=4\Tilde{Q}_{\omega,c}(\phi,\psi,\varphi)-9\Tilde{V}(\phi,\psi,\varphi)=-2\Tilde{Q}_{\omega,c}(\phi,\psi,\varphi)<0.
 \end{align}
 By the Lagrange multiplier theorem, there exists $\lambda\in\mathbb{R}$ such that $\Tilde{S}^{\prime}_{\omega,c}(\phi,\psi,\varphi)=\lambda \Tilde{N}^{\prime}_{\omega,c}(\phi,\psi,\varphi)$. Moreover, we have
 \begin{align}\label{exist:2}
     \lambda\left(\Tilde{N}^{\prime}_{\omega,c}(\phi,\psi,\varphi),(\phi,\psi,\varphi)\right)=\left(\Tilde{S}^{\prime}_{\omega,c}(\phi,\psi,\varphi),(\phi,\psi,\varphi)\right)=\Tilde{N}_{\omega,c}(\phi,\psi,\varphi)=0.
 \end{align}
Combining \eqref{exist:1} and \eqref{exist:2}, we can obtain $\lambda=0$. Hence, $\Tilde{S}^{\prime}_{\omega,c}(\phi,\psi,\varphi)=0$, which implies $(\phi,\psi,\varphi)\in \Tilde{\mathcal{A}}_{\omega,c}$.

Notice that $\Tilde{\mathcal{A}}_{\omega,c}\subset \Tilde{\mathcal{N}}_{\omega,c}$ and $(\phi,\psi,\varphi)\in \Tilde{a}_{\omega,c}$, we have
\begin{align*}
    \Tilde{S}_{\omega,c}(\phi,\psi,\varphi)\leq \Tilde{S}_{\omega,c}(\phi_1,\psi_1,\varphi_1)~~\text{for all}~~(\phi_1,\psi_1,\varphi_1)\in\Tilde{\mathcal{A}}_{\omega,c},
\end{align*}
which implies $(\phi,\psi,\varphi)\in \Tilde{\mathcal{G}}_{\omega,c}$.
 \end{proof}

 \begin{lemma}\label{lemma:b}
 If \eqref{D:condition}  holds and $\Tilde{a}_{\omega,c}\neq \emptyset$, then $\widetilde{\mathcal{G}}_{\omega,c}\subset\Tilde{a}_{\omega,c}$.
 \end{lemma}
\begin{proof}
Let $(\phi,\psi,\varphi)\in \widetilde{\mathcal{G}}_{\omega,c} $. One can take $(\phi_1,\psi_1,\varphi_1)\in \Tilde{a}_{\omega,c}$, by Lemma \ref{lemma:a}, we have $(\phi_1,\psi_1,\varphi_1)\in \widetilde{\mathcal{G}}_{\omega,c}$, i.e., $\Tilde{S}_{\omega,c}(\phi_1,\psi_1,\varphi_1)=\Tilde{S}_{\omega,c}(\phi,\psi,\varphi)$. Therefore, for each $(u,v,w)\in\Tilde{\mathcal{N}}_{\omega,c}$, we obtain
\begin{align*}
    \Tilde{S}_{\omega,c}(\phi,\psi,\varphi)=\Tilde{S}_{\omega,c}(\phi_1,\psi_1,\varphi_1)\leq \Tilde{S}_{\omega,c}(u,v,w).
\end{align*}
Since $(\phi,\psi,\varphi)\in \widetilde{\mathcal{G}}_{\omega,c}\subset\Tilde{\mathcal{N}}_{\omega,c}$, we deduce that $(\phi,\psi,\varphi)\in \Tilde{a}_{\omega,c}$.
\end{proof}

\begin{lemma}\label{lemma:positive}
If \eqref{D:condition} holds, then $\Tilde{\mu}_{\omega,c}>0$.
\end{lemma}
\begin{proof}
By the definition of $\Tilde{S}_{\omega,c}$,  $\Tilde{Q}_{\omega,c}$ and $\Tilde{V}$, we have the following relation
\[\Tilde{S}_{\omega,c}=\frac{1}{3}\Tilde{Q}_{\omega,c}+\frac{1}{3}\Tilde{N}_{\omega,c}.\]
One can rewrite  $\Tilde{\mu}_{\omega,c}$ as
\begin{align}\label{exist:3}
    \Tilde{\mu}_{\omega,c}=\inf\left\{\frac{1}{3}\Tilde{Q}_{\omega,c},~~(\phi,\psi,\varphi)\in\Tilde{\mathcal{N}_{\omega,c}}\right\}.
\end{align}
We claim
\begin{align}\label{exist:claim}
    \Tilde{Q}_{\omega,c}(u,v,w)\leq C \Tilde{Q}_{\omega,c}(u,v,w)^{\frac{3}{2}}.
\end{align}
If claim \eqref{exist:claim} is true, then by dividing the both sides of \eqref{exist:claim} by $\Tilde{Q}_{\omega,c}>0$, we obtain the positive lower bound for $\Tilde{Q}_{\omega,c}$. Hence, combining with \eqref{exist:3}, we can obtain the desired result.

Now we prove the claim.  Notice that for $(u,v,w)\in\Tilde{N}_{\omega,c}$, we have $2\Tilde{Q}_{\omega,c}(u,v,w)=3\Tilde{V}(u,v,w)$.

{\bf Case (A).} By H\"older inequality and Sobolev embedding theorem, we have
\begin{align*}
    2\Tilde{Q}_{\omega,c}(u,v,w)=3\Tilde{V}(u,v,w)&\leq C\|u\|_{L^3}\|v\|_{L^3}\|w\|_{L^3}\\
    &\leq C \|u\|_{H^1}^{\frac{1}{2}}\|v\|_{H^1}^{\frac{1}{2}}\|w\|_{H^1}^{\frac{1}{2}}\leq C \Tilde{Q}_{\omega,c}(u,v,w)^{\frac{3}{2}}.
\end{align*}

{\bf Case (B).} Since $3\leq N\leq 5$, by H\"older inequality, we have
\begin{align*}
   2\Tilde{Q}_{\omega,c}=3\Tilde{V}(u,v,w)\leq & C\|u\|_{\frac{4N}{N+2}}\|v\|_{\frac{4N}{N+2}}\|w\|_{L^{\frac{2N}{N-2}}}\leq C\|u\|_{H^1}^{\frac{1}{2}}\|v\|_{H^1}^{\frac{1}{2}}\|\nabla w\|_{L^2} \\
   \leq&C\Tilde{Q}_{\omega,c}(u,v,w)^{\frac{3}{2}},
\end{align*}
where we used the Sobolev embedding since $\frac{4N}{N+2}\in\left[2,\frac{2N}{N-2}\right]$.

{\bf Case (C).} Notice that $4\leq N\leq 5$, we have
\begin{align*}
    2\Tilde{Q}_{\omega,c}(u,v,w)=3\Tilde{V}(u,v,w)&\leq C\|u\|_{L^{\frac{N}{2}}} \|v\|_{L^{\frac{2N}{N-2}}}\|w\|_{L^{\frac{2N}{N-2}}}\\
    &\leq C\|u\|_{H^1}^{\frac{1}{2}}\|\nabla v\|_{L^2}\| \nabla w\|_{L^2}\leq C \Tilde{Q}_{\omega,c}(u,v,w)^{\frac{3}{2}}
\end{align*}
where we used the Sobolev embedding since $\frac{N}{2}\in\left[2,\frac{2N}{N-2}\right]$.

By the similar argument as the cases $(B)$ and $(C)$, we can obtain the cases $(D)$ and $(E)$.

{\bf Case (D).} Notice that $3\leq N\leq 5$, we have
\begin{align*}
    2\Tilde{Q}_{\omega,c}(u,v,w)=3\Tilde{V}(u,v,w)&\leq C \|u\|_{L^{\frac{4N}{N+2}}}\|v\|_{L^{\frac{2N}{N-2}}}\|w\|_{L^{\frac{4N}{N+2}}}\\
    &\leq C\|u\|_{H^1}^{\frac{1}{2}}\|\nabla v\|_{L^2}\| w\|_{H^1}^{\frac{1}{2}}\leq C \Tilde{Q}_{\omega,c}(u,v,w)^{\frac{3}{2}}.
\end{align*}

{\bf Case (E).} Notice that $4\leq N\leq 5$, we have
\begin{align*}
    2\Tilde{Q}_{\omega,c}(u,v,w)=3\Tilde{V}(u,v,w)&\leq C \|u\|_{L^{\frac{2N}{N-2}}}\|v\|_{L^{\frac{2N}{N-2}}}\|w\|_{L^{\frac{N}{2}}}\\
    &\leq C\|\nabla u\|_{L^2}\|\nabla  v\|_{L^2}\| w\|_{H^1}^{\frac{1}{2}}\leq CC \Tilde{Q}_{\omega,c}(u,v,w)^{\frac{3}{2}}.
\end{align*}

Combining the above five cases, we can obtain the Claim \eqref{exist:claim}. Hence, we complete the proof of this Lemma.
\end{proof}

\begin{lemma}\label{lemma:upper}
Assume \eqref{D:condition}. If $(u,v,w)\in \Tilde{X}_{\omega,c}$ satisfies $\Tilde{N}_{\omega,c}(u,v,w)<0$, then $\frac{1}{3}\Tilde{Q}_{\omega,c}>\Tilde{\mu}_{\omega,c}$.
\end{lemma}
\begin{proof}
If $\Tilde{N}_{\omega,c}(u,v,w)<0$, then by the definition of $\Tilde{N}_{\omega,c}$, we have $3\Tilde{V}(u,v,w)>2\Tilde{Q}_{\omega,c}(u,v,w)>0$. From this, we have
\begin{align*}
    \lambda_0:=\frac{2\Tilde{Q}_{\omega,c}(u,v,w)}{3\Tilde{V}(u,v,w)}\in(0,1)
\end{align*}
and $\Tilde{N}_{\omega,c}(\lambda_0u,\lambda_0v,\lambda_0w)=0$. By \eqref{exist:3}, we deduce
\begin{align*}
    \Tilde{\mu}_{\omega,c}\leq\frac{1}{3}\Tilde{Q}_{\omega,c}(\lambda_0u,\lambda_0v,\lambda_0w)=\frac{\lambda_0^2}{3}\Tilde{Q}_{\omega,c}(u,v,w)<\frac{1}{3}\Tilde{Q}_{\omega,c}(u,v,w).
\end{align*}
Now we complete the proof this Lemma.
\end{proof}

\begin{lemma}\label{lemma:nonlinear}
Assume \eqref{D:condition}. If the sequence $\{(u_n,v_n,w_n)\}$ weakly converges to $(u,v,w)\in\Tilde{X}_{\omega,c}$, then
\begin{align*}
    \Tilde{V}(u_n,v_n,w_n)-\Tilde{V}(u_n-u,v_n-v,w_n-w)\to \Tilde{V}(u,v,w)~~\text{as}~~n\to\infty.
\end{align*}
\end{lemma}
\begin{proof}
By the direct calculation, we have
\begin{align*}
    &\Tilde{V}(u_n,v_n,w_n)-\Tilde{V}(u_n-u,v_n-v,w_n-w)- \Tilde{V}(u,v,w)\\
    =&\Re\int e^{i\left(\frac{\gamma_1+\gamma_2}{2}-\frac{\gamma_3}{2}\right)c\cdot x}\left((u_nv+uv_n-uv)\bar{w}_n+(u_nv_n-u_nv-uv_n)\bar{w}\right)dx\\
    =&\Re\int e^{i\left(\frac{\gamma_1+\gamma_2}{2}-\frac{\gamma_3}{2}\right)}(u_nv\bar{w}_n+uv_n\bar{w}_n-uv\bar{w}_n+u_nv_n\bar{w}-u_nv\bar{w}-uv_n\bar{w}).
\end{align*}
We aim to prove the right-hand side  vanishes as $n\to\infty$.

{\bf Case (A).}  From the Sobolev embedding, we have, as $n\to\infty$,
\begin{align*}
    &(u_n,v_n,w_n)\rightharpoonup (u,v,w)~~\text{in}~~L^3(\mathbb{R}^N)\times L^3(\mathbb{R}^N)\times L^3(\mathbb{R}^N),\\
    &u_nv_n\rightharpoonup uv~~\text{in}~~L^\frac{3}{2}(\mathbb{R}^N),~~u_n\bar{w}_n\rightharpoonup u\bar{w},~~\text{in}~~L^\frac{3}{2}(\mathbb{R}^N),~~v_n\bar{w}_n\rightharpoonup v\bar{w},~~\text{in}~~L^\frac{3}{2}(\mathbb{R}^N).
\end{align*}

{\bf Case (B).} As $n\to\infty$,
\begin{align*}
    &(u_n,v_n,w_n)\rightharpoonup (u,v,w)~~\text{in}~~L^{\frac{4N}{N+2}}(\mathbb{R}^N)\times L^{\frac{4N}{N+2}}(\mathbb{R}^N)\times L^{\frac{2N}{N-2}}(\mathbb{R}^N),\\
    &u_nv_n\rightharpoonup uv~~\text{in}~~L^\frac{2N}{N+2}(\mathbb{R}^N),~~u_n\bar{w}_n\rightharpoonup u\bar{w},~~\text{in}~~L^\frac{4N}{3N-2}(\mathbb{R}^N),~~v_n\bar{w}_n\rightharpoonup v\bar{w},~~\text{in}~~L^\frac{4N}{3N-2}(\mathbb{R}^N).
\end{align*}

{\bf Case (C).} As $n\to\infty$,
\begin{align*}
    &(u_n,v_n,w_n)\rightharpoonup (u,v,w)~~\text{in}~~L^{\frac{N}{2}}(\mathbb{R}^N)\times L^{\frac{2N}{N-2}}(\mathbb{R}^N)\times L^{\frac{2N}{N-2}}(\mathbb{R}^N),\\
    &u_nv_n\rightharpoonup uv~~\text{in}~~L^\frac{2N}{N+2}(\mathbb{R}^N),~~u_n\bar{w}_n\rightharpoonup u\bar{w},~~\text{in}~~L^\frac{2N}{N+2}(\mathbb{R}^N),~~v_n\bar{w}_n\rightharpoonup v\bar{w},~~\text{in}~~L^\frac{N}{N-2}(\mathbb{R}^N).
\end{align*}

{\bf Case (D).} As $n\to\infty$,
\begin{align*}
    &(u_n,v_n,w_n)\rightharpoonup (u,v,w)~~\text{in}~~L^{\frac{4N}{N+2}}(\mathbb{R}^N)\times L^{\frac{2N}{N-2}}(\mathbb{R}^N)\times L^{\frac{4N}{N+2}}(\mathbb{R}^N),\\
    &u_nv_n\rightharpoonup uv~~\text{in}~~L^\frac{4N}{3N-2}(\mathbb{R}^N),~~u_n\bar{w}_n\rightharpoonup u\bar{w},~~\text{in}~~L^\frac{2N}{N+2}(\mathbb{R}^N),~~v_n\bar{w}_n\rightharpoonup v\bar{w},~~\text{in}~~L^\frac{4N}{3N-2}(\mathbb{R}^N).
\end{align*}

{\bf Case (E).} As $n\to\infty$,
\begin{align*}
    &(u_n,v_n,w_n)\rightharpoonup (u,v,w)~~\text{in}~~ L^{\frac{2N}{N-2}}(\mathbb{R}^N)\times L^{\frac{2N}{N-2}}(\mathbb{R}^N)\times L^{\frac{N}{2}}(\mathbb{R}^N),\\
    &u_nv_n\rightharpoonup uv~~\text{in}~~L^\frac{N}{N-2}(\mathbb{R}^N),~~u_n\bar{w}_n\rightharpoonup u\bar{w},~~\text{in}~~L^\frac{2N}{N+2}(\mathbb{R}^N),~~v_n\bar{w}_n\rightharpoonup v\bar{w},~~\text{in}~~L^\frac{2N}{N+2}(\mathbb{R}^N).
\end{align*}

Then by the H\"older inequality and above weak convergences, we can easily obtain the desired result.
\end{proof}

Let
\begin{align*}
    \Tilde{\tau}_{y}(u,v,w):=\left(e^{-\frac{i\gamma_1}{2}c\cdot x}u(\cdot-y),e^{-\frac{i\gamma_2}{2}c\cdot x}v(\cdot-y),e^{-i\frac{\gamma_3}{2}c\cdot x}w(\cdot-y)\right),
\end{align*}
then we have
\[\Tilde{Q}_{\omega,c}( \Tilde{\tau}_{y}(u,v,w))=\Tilde{Q}_{\omega,c}(u,v,w),~~\Tilde{V}( \Tilde{\tau}_{y}(u,v,w))=\Tilde{V}(u,v,w),\]
for all $y\in\mathbb{R}^N$.

\begin{lemma}\label{lemma:w}
Assume \eqref{D:condition} holds. Further more, we also assume that $N=5$ for the cases $(C)$ and $(E)$. If a sequence $\{(u_n,v_n,w_n)\}$ in $\Tilde{X}_{\omega,c}$ satisfies
\begin{align*}
    \Tilde{Q}_{\omega,c}(u_n,v_n,w_n)\to A_1,~~\Re\int e^{i\left(\frac{\gamma_1+\gamma_2}{3}-\frac{\gamma_3}{2}\right)c\cdot x}u_nv_n\bar{w}_n\to A_2~~\text{as}~~n\to\infty,
\end{align*}
for some positive constants $A_1,A_2>0$, then there exist $\{y_n\}$ and $(u,v,w)\in\Tilde{X}_{\omega,c}\setminus \{(0,0,0)\}$ such that $\{\Tilde{\tau}_{y}(u_n,v_n,w_n)\}$ has a subsequence that weakly converges to $(u,v,w)$ in $\Tilde{X}_{\omega,c}$.
\end{lemma}
\begin{proof}
Since $\Tilde{Q}_{\omega,c}(u_n,v_n,w_n)\to A_1$, we deduce that the sequence $\{(u_n,v_n,w_n)\}$ is bounded in $\Tilde{X}_{\omega}$. Since $\Re\int e^{i\left(\frac{\gamma_1+\gamma_2}{2}-\frac{\gamma_3}{2}\right)c\cdot x}u_nv_n\bar{w}_n\to A_2$, we obtain
\begin{align*}
    &\textbf{Case (A)}~~\limsup_{n\to\infty}\|u_n\|_{L^3}>0,~~\limsup_{n\to\infty}\|v_n\|_{L^3}>0,~~\limsup_{n\to\infty}\|w_n\|_{L^3}>0,\\
    &\textbf{Case (B)}~~\limsup_{n\to\infty}\|u_n\|_{L^{\frac{4N}{N+2}}}>0,~~\limsup_{n\to\infty}\|v_n\|_{L^{\frac{4N}{N+2}}}>0,\\
    &\textbf{Case (C)}~~\limsup_{n\to\infty}\|uu_n\|_{L^{\frac{N}{2}}}>0,\\
    &\textbf{Case (D)}~~\limsup_{n\to\infty}\|u_n\|_{L^{\frac{4N}{N+2}}}>0,~~\limsup_{n\to\infty}\|w_n\|_{L^{\frac{4N}{N+2}}}>0,\\
    &\textbf{Case (E)}~~\limsup_{n\to\infty}\|w_n\|_{L^{\frac{N}{2}}}>0.
\end{align*}
Hence, by \cite[Lemma 6]{L1983Invent}, we can obtain the desired result.
\end{proof}

\begin{lemma}\label{lemma:strong}
Assume \eqref{D:condition}. Further more, we also assume that $N=5$ for the cases $(C)$ and $(E)$. If a sequence  $\{(u_n,v_n,w_n)\}$ in $\Tilde{X}_{\omega,c}$ satisfies
\begin{align*}
    \Tilde{N}_{\omega,c}(u_n,v_n,w_n)\to0,~~\Tilde{S}_{\omega,c}(u_n,v_n,w_n)\to\Tilde{\mu}_{\omega,c}~~\text{as}~~n\to\infty,
\end{align*}
then there exist $\{y_n\}$ and $(u,v,w)\in\Tilde{X}_{\omega,c}\setminus\{(0,0,0)\}$ such that $\{\Tilde{\tau}(u_n,v_n,w_n)\}$ has a subsequence that converges to $(u,v,w)$ in $\Tilde{X}_{\omega,c}$. In particular, $(u,v,w)\in\Tilde{a}_{\omega,c}$.
\end{lemma}
\begin{proof}
By assumptions, we have
\begin{align*}
    \frac{1}{3}\Tilde{Q}_{\omega,c}(u_n,v_n,w_n)=\Tilde{S}_{\omega,c}(u_n,v_n,w_n)-\frac{1}{3}\Tilde{V}(u_n,v_n,w_n)\Tilde{\mu}_{\omega,c},\\
    \frac{1}{2}\Tilde{V}(u_n,v_n,w_n)=\Tilde{S}_{\omega,c}(u_n,v_n,w_n)-\frac{1}{2}\Tilde{N}_{\omega,c}(u_n,v_n,w_n)\to \Tilde{\mu}_{\omega,c}.
\end{align*}
By Lemma \ref{lemma:positive} and Lemma \ref{lemma:w}, there exist $\{y_n\}$, $(u,v,w)\in \Tilde{X}_{\omega,c}\setminus\{(0,0,0)\}$, and a subsequence of $\Tilde{\tau}_{y_n}(u_n,v_n,w_n)$ (still denoted by $\Tilde{\tau}_{y_n}(u_n,v_n,w_n)$) such that $\Tilde{\tau}_{y_n}(u_n,v_n,w_n)\rightharpoonup (u,v,w)$ weakly in $\Tilde{X}_{\omega,c}$.

By the weakly convergence of $\Tilde{\tau}_{y_n}(u_n,v_n,w_n)$ and Lemma \ref{lemma:nonlinear}, we have
\begin{align}\label{conv:1}
    \Tilde{Q}_{\omega,c}(\Tilde{\tau}_{y_n}(u_n,v_n,w_n))-\Tilde{Q}_{\omega,c}\left(\Tilde{\tau}_{y_n}(u_n,v_n,w_n)-(u,v,w)\right)\to\Tilde{Q}_{\omega,c}(u,v,w),\\\label{conv:2}
    \Tilde{N}_{\omega,c}(\Tilde{\tau}_{y_n}(u_n,v_n,w_n))-\Tilde{N}_{\omega,c}\left(\Tilde{\tau}_{y_n}(u_n,v_n,w_n)-(u,v,w)\right)\to\Tilde{N}_{\omega,c}(u,v,w).
\end{align}
From \eqref{conv:1} and $\Tilde{Q}_{\omega,c}(u,v,w)>0$, we obtain that, up to subsequence,
\begin{align*}
    \frac{1}{3}\lim_{n\to\infty}\Tilde{Q}_{\omega,c}\left(\Tilde{\tau}_{y_n}(u_n,v_n,w_n)-(u,v,w)\right)<\frac{1}{3}\lim_{n\to\infty}\Tilde{Q}_{\omega,c}(\Tilde{\tau}_{y_n}(u_n,v_n,w_n))=\Tilde{\mu}_{\omega,c}.
\end{align*}
From this and Lemma \ref{lemma:upper}, we obtain $\Tilde{N}_{\omega,c}\left(\Tilde{\tau}_{y_n}(u_n,v_n,w_n)-(u,v,w)\right)>0$ for $n$ large enough. Therefore, from \eqref{conv:2}, $\Tilde{N}_{\omega,c}(u,v,w)\leq0$ since $\Tilde{N}_{\omega,c}(\Tilde{\tau}_{y_n}(u_n,v_n,w_n))\to0$. Again, by Lemma \ref{lemma:upper} and the weakly semi-continuity of norms, we have
\begin{align*}
    \Tilde{\mu}_{\omega,c}\leq\frac{1}{3}\Tilde{Q}_{\omega,c}(u,v,w)\leq\frac{1}{3}\Tilde{Q}_{\omega,c}\left(\Tilde{\tau}_{y_n}(u_n,v_n,w_n)\right)=\Tilde{\mu}_{\omega,c}.
\end{align*}
Therefore, by \eqref{conv:1}, we deduce $\Tilde{Q}_{\omega,c}\left(\Tilde{\tau}_{y_n}(u_n,v_n,w_n)-(u,v,w)\right)\to0$ as $n\to\infty$, which implies that $\Tilde{\tau}_{y_n}(u_n,v_n,w_n)\to(u,v,w)$ strongly in $\Tilde{X}_{\omega,c}$. Now we  complete the proof of this Lemma.
\end{proof}

\begin{proof}[\bf Proof of Theorem \ref{lemma:nonempty}.]
Combining the Lemmas \ref{lemma:a}, \ref{lemma:b} and \ref{lemma:strong}, we can get Theorem \ref{lemma:nonempty}.
\end{proof}

\section{Nonexistence result}
In this section, we show that nonexistence of non-trivial solutions for \eqref{equ:s} with $\gamma_1+\gamma_2=\gamma_2$ and $\omega=\frac{\gamma_1|c|^2}{4}=\frac{\gamma_2|c|^2}{4}$.

\begin{proof}[\bf Proof of Theorem \ref{Thm2} ]
Let $(\Tilde{\phi},\Tilde{\psi},\Tilde{\varphi})=\left(e^{-\frac{i}{2}c\cdot x}\phi,e^{-\frac{i}{2}c\cdot x}\psi,e^{-\frac{i\gamma_3}{2}c\cdot x}\varphi\right)$. By $\gamma_1+\gamma_2=\gamma_2$ and $\omega=\frac{\gamma_1|c|^2}{4}=\frac{\gamma_2|c|^2}{4}$, then $(\Tilde{\phi},\Tilde{\psi},\Tilde{\varphi})$ is a solution of the following system
\begin{align*}
    \begin{cases}
    -\Delta\Tilde{\phi}-\Tilde{\varphi}\bar{\Tilde{\psi}} =0,\\
    -\Delta\Tilde{\psi}-\Tilde{\varphi}\bar{\Tilde{\phi}} =0,\\
    -\Delta\Tilde{\varphi}-\Tilde{\phi}\Tilde{\psi}=0.\\
    \end{cases}
\end{align*}
This is equivalent to $E^{\prime}(\Tilde{\phi},\Tilde{\psi},\Tilde{\varphi})=0$. Therefore, we have
\begin{align}\label{non:1}
    0=\left(E^\prime(\Tilde{\phi},\Tilde{\psi},\Tilde{\varphi}),(\Tilde{\phi},\Tilde{\psi},\Tilde{\varphi})\right)=K(\phi,\psi,\varphi)-3\Re\int \Tilde{\phi}\Tilde{\psi}\bar{\Tilde{\varphi}}dx.
\end{align}
Let $u_{\lambda}(x)=\lambda^{\frac{N}{2}}u(\lambda x)$ for $\lambda>0$, and we have
\begin{align}\label{non:2}
    0=&\left(E^\prime(\Tilde{\phi},\Tilde{\psi},\Tilde{\varphi}),\partial_{\lambda}(\Tilde{\phi}_{\lambda},\Tilde{\psi}_{\lambda},\Tilde{\varphi}_{\lambda})|_{\lambda=1}\right)=\partial_{\lambda}E^{\prime}(\Tilde{\phi}_{\lambda},\Tilde{\psi}_{\lambda},\Tilde{\varphi}_{\lambda})|_{\lambda=1}\notag\\
    =&\partial_{\lambda}\left(\frac{\lambda^2}{2}K(\Tilde{\phi},\Tilde{\psi},\Tilde{\varphi})-\lambda^{\frac{N}{2}}\Re\int \Tilde{\phi}\Tilde{\psi}\bar{\Tilde{\varphi}}\right)\big|_{\lambda=1}\notag\\
    =&K(\Tilde{\phi},\Tilde{\psi},\Tilde{\varphi})-\frac{N}{2}\Re\int\Tilde{\phi}\Tilde{\psi}\bar{\Tilde{\varphi}}.
\end{align}
Then \eqref{non:1} and \eqref{non:2} yields
\begin{align*}
    (N-6)K(\Tilde{\phi},\Tilde{\psi},\Tilde{\varphi})=0.
\end{align*}
Since $N<6$, we obtain that $(\Tilde{\phi},\Tilde{\psi},\Tilde{\varphi})=(0,0,0)$. This completes the proof.
\end{proof}

\section{Global existence}
In this section, we aim to obtain the global existence and proof the Theorem \ref{Thm3}. First, we introduce the subsets of the energy space
\begin{align*}
    A^{+}_{\omega,c}=\{(u,v,w)\in H^1(\mathbb{R}^N)\times H^1(\mathbb{R}^N)\times H^1(\mathbb{R}^N)~:~S_{\omega,c}(u,v,w)\leq \mu_{\omega,c},~N_{\omega,c}(u,v,w)\geq0\},\\
    A^{-}_{\omega,c}=\{(u,v,w)\in H^1(\mathbb{R}^N)\times H^1(\mathbb{R}^N)\times H^1(\mathbb{R}^N)~:~S_{\omega,c}(u,v,w)\leq \mu_{\omega,c},~N_{\omega,c}(u,v,w)<0\}.
\end{align*}
Now we show that $A^{\pm}_{\omega,c}$ are invariant sets under the flow.
\begin{lemma}\label{lemma:invariant}
Assume \eqref{D:condition}. Then the sets $A^{\pm}_{\omega,c}$ are invariant under the flow of \eqref{equ:system}.
\end{lemma}
\begin{proof}
Let $(u_0,v_0,w_0)\in A^+_{\omega,c}$. It is obvious that $S_{\omega,c}(u(t),v(t),w(t))\leq \mu_{\omega,c}$ for all $t\in I$, where $I$ is the maximal existence interval of $H^1$-solution, since the corresponding  mass, energy and momentum are conserved.

Now we show that $N_{\omega,c}(u(t),v(t),w(t))\geq0$ for all $t\in I$. If not, there exist $t_1,t_2\in I$ such that $N_{\omega,c}(u(t_1),v(t_1),w(t_1))<0$ and $N_{\omega,c}(u(t_2),v(t_2),w(t_2))=0$. By the uniqueness of Cauchy problem for \eqref{equ:system}, we have $(u(t_2),v(t_2),w(t_2))\neq(0,0,0)$. Moreover, from $S_{\omega,c}(u(t),v(t),w(t))\leq \mu_{\omega,c}$, we obtain $(u(t_2),v(t_2),w(t_2))\in a_{\omega,c}\subset\mathcal{G}_{\omega,c}$. This yields that
\begin{align*}
    &(u(t),v(t),w(t))\\
    =&\left(e^{i\omega (t-t_2)}u(t_2,x-c(t-t_2)),e^{i\omega (t-t_2)}v(t_2,x-c(t-t_2)),e^{2i\omega (t-t_2)}w(t_2,x-c(t-t_2))\right)
\end{align*}
for all $t\in\mathbb{R}$. In particular, $N_{\omega,c}(u(t),v(t),w(t))=0$ for all $t\in \mathbb{R}$, which contradicts $N_{\omega,c}(u(t_1),v(t_1),w(t_1))<0$.

By the similar argument as above, we can proof the case $A^-_{\omega,c}$. Now we complete the proof of this lemma.
\end{proof}

\begin{lemma}\label{lemma:global}
Assume \eqref{D:condition}. If the initial data $(u_0,v_0,w_0)\in A^+_{\omega,c}$, then the $H^1$ solution $(u(t),v(t),w(t))$ of \eqref{equ:system} exists globally in time and
\begin{align*}
    \sup_{t\in\mathbb{R}}\|(u(t),v(t),w(t))\|_{H^1\times H^1\times H^1}\leq C(\|(u_0,v_0,w_0)\|_{H^1\times H^1\times H^1})<\infty.
\end{align*}
\end{lemma}
\begin{proof}
Combining Lemma \ref{lemma:invariant} and the definition of $S_{\omega,c}$, $Q_{\omega,c}$ and $N_{\omega,c}$, we have
\begin{align*}
    \mu_{\omega,c}
    >&S_{\omega,c}(u(t),v(t),w(t))\\
    =&\frac{1}{3}Q_{\omega,c}(u(t),v(t),w(t))+\frac{1}{3}N_{\omega,c}(u(t),v(t),w(t))\\
    \geq& \frac{1}{3}Q_{\omega,c}(u(t),v(t),w(t))\\
    \geq&\frac{1}{6}\left\|\nabla u(t)-\frac{i\gamma_1}{2}cu(t)\right\|_{L^2}^2+\frac{1}{6}\left\|\nabla v(t)-\frac{i\gamma_2}{2}cv(t)\right\|_{L^2}^2+\frac{1}{6}\left\|\nabla w(t)-\frac{i\gamma_3}{2}cw(t)\right\|_{L^2}^2.
\end{align*}
By the conservation of  mass, we deduce
\begin{align*}
    &\|\nabla u(t)\|_{L^2}^2+\|\nabla v(t)\|_{L^2}^2+\|\nabla w(t)\|_{L^2}^2\\
    \leq&C\left(\left\|\nabla u(t)-\frac{i\gamma_1}{2}cu(t)\right\|_{L^2}^2+\left\|\nabla v(t)-\frac{i\gamma_2}{2}cv(t)\right\|_{L^2}^2+\left\|\nabla w(t)-\frac{i\gamma_3}{2}cw(t)\right\|_{L^2}^2\right)\\
    &+C\left(\|u(t)\|_{L^2}^2+\|v(t)\|_{L^2}^2+\|w(t)\|_{L^2}^2\right)\\
    \leq&C\mu_{\omega,c}+CM(u_0,v_0,w_0).
\end{align*}
This means that $H^1$ solution is globally in time.
Now we complete the proof of this lemma.
\end{proof}

\begin{lemma}\label{lemma:scaling}
Let $(\omega,c)\in \mathbb{R}\times\mathbb{R}^4$ satisfy $\omega>0$ and $c\neq0$. Then
\begin{align*}
    \mu_{\omega,c}=|c|^{2}\mu_{\frac{\omega}{|c|^2},\frac{c}{|c|}}.
\end{align*}
\end{lemma}
\begin{proof}
Let $\{(u_n,v_n,w_n)\}$ be a minimizing sequence for $\mu_{\omega,c}$, that is,
\begin{align*}
    N_{\omega,c}(u_n,v_n,w_n)=0,~~S_{\omega,c}\to\mu_{\omega,c}.
\end{align*}
Let $(\Tilde{u}_n,\Tilde{v}_n,\Tilde{w}_n)=\left(\frac{1}{|c|^2}u_n\left(\frac{x}{|c|}\right),\frac{1}{|c|^2}v_n\left(\frac{x}{|c|}\right),\frac{1}{|c|^2}w_n\left(\frac{x}{|c|}\right)\right)$. Then
\begin{align*}
    &N_{\frac{\omega}{|c|^2},\frac{c}{|c|}}(\Tilde{u}_n,\Tilde{v}_n,\Tilde{w}_n)=|c|^{2}N_{\omega,c}(u_n,v_n,w_n)=0,\\
     &S_{\frac{\omega}{|c|^2},\frac{c}{|c|}}(\Tilde{u}_n,\Tilde{v}_n,\Tilde{w}_n)=|c|^{2}S_{\omega,c}\to|c|^{2}\mu_{\omega,c}.
\end{align*}
Hence, we can obtain the desired result.
\end{proof}

\begin{proof}[{\bf Proof of Theorem \ref{Thm3}}]
To prove (i).
Let $A_0$ to be chosen later. We claim  that if $(u_0,v_0,w_0)\in H^1(\mathbb{R}^4)$ satisfies $\max\{\|u_0\|_{L^2}^2,\|v_0\|_{L^2}^2\}<A_0$, then we claim that
\begin{align}\label{global:1}
    &S_{\frac{\gamma_3|c|^2}{8},c}\left(e^{\frac{i\gamma_1}{2}c\cdot x}u_0,e^{\frac{i\gamma_2}{2}c\cdot x}v_0,e^{i\frac{\gamma_3}{2}c\cdot x}w_0\right)<\mu_{\frac{\gamma_3|c|^2}{8},c},\\\label{global:2}
    &N_{\frac{\gamma_3|c|^2}{8},c}\left(e^{\frac{i\gamma_1}{2}c\cdot x}u_0,e^{\frac{i\gamma_2}{2}c\cdot x}v_0,e^{i\frac{\gamma_3}{2}c\cdot x}w_0\right)\geq0.
\end{align}
for large $|c|$.  If this claim is true. Then we can obtain
\begin{align*}
    \left(e^{\frac{i\gamma_1}{2}c\cdot x}u_0,e^{\frac{i\gamma_2}{2}c\cdot x}v_0,e^{i\frac{\gamma_3}{2}c\cdot x}w_0\right)\in A^+_{\omega,c}
\end{align*}
for $|c|$ large enough. Hence, by Lemma \ref{lemma:global}, we can obtain the desired result.

Now we prove the claim. Indeed, by the definition of $Q_{\omega,c}$, we have
\begin{align*}
    Q_{\frac{\gamma_3|c|^2}{8},c}(u,v,w)=&\frac{1}{2}K(u,v,w)+\frac{\gamma_1\gamma_3|c|^2}{16}\|u\|_{L^2}^2\\
    &+\frac{\gamma_2\gamma_3|c|^2}{16}\|v\|_{L^2}^2+\frac{\gamma_3^2|c|^2}{8}\|w\|_{L^2}^2+\frac{1}{2}c\cdot P(u,v,w)\\
    =&\frac{1}{2}\left\|\nabla \left(e^{-\frac{i\gamma_1}{2}c\cdot x}u\right)\right\|_{L^2}^2+\frac{1}{2}\left\|\nabla \left(e^{-\frac{i\gamma_2}{2}c\cdot x}v\right)\right\|_{L^2}^2+\frac{1}{2}\left\|\nabla \left(e^{-\frac{i\gamma_3}{2}c\cdot x}w\right)\right\|_{L^2}^2\\
    &+|c|^2\frac{\gamma_1(\gamma_3-2\gamma_1)}{16}\|u\|_{L^2}^2+|c|^2\frac{\gamma_2(\gamma_3-2\gamma_2)}{16}\|v\|_{L^2}^2.
\end{align*}
From Lemma \ref{lemma:scaling} and above, \eqref{global:1} is equivalent to
\begin{align}\label{global:3}
    &S_{\frac{\gamma_3|c|^2}{8},c}\left(e^{\frac{i\gamma_1}{2}c\cdot x}u_0,e^{\frac{i\gamma_2}{2}c\cdot x}v_0,e^{i\frac{\gamma_3}{2}c\cdot x}w_0\right)\notag\\
    =&\frac{1}{2}K(u_0,v_0,w_0)-V\left(e^{\frac{i\gamma_1}{2}c\cdot x}u_0,e^{\frac{i\gamma_2}{2}c\cdot x}v_0,e^{i\frac{\gamma_3}{2}c\cdot x}w_0\right)\notag\\
    &+|c|^2\frac{\gamma_1(\gamma_3-2\gamma_1)}{16}\|u_0\|_{L^2}^2+|c|^2\frac{\gamma_2(\gamma_3-2\gamma_2)}{16}\|v_0\|_{L^2}^2\notag\\
    \leq&|c|^2\mu_{\frac{\gamma_3}{8},\frac{c}{|c|}}.
\end{align}
Let
\begin{align*}
    A_0=\frac{16}{2\max\{\gamma_1,\gamma_2\}(\gamma_3-\gamma_1-\gamma_2)}\mu_{\frac{\gamma_3}{8},\frac{c}{|c|}}.
\end{align*}
Then by $\gamma_3>\gamma_1+\gamma_2$ and  Lemma \ref{lemma:b}, we get $A_0>0$. Hence,
\begin{align*}
    |c|^2\left(\mu_{\frac{\gamma_3}{8},\frac{c}{|c|}}-\frac{\gamma_1(\gamma_3-2\gamma_1)}{16}\|u_0\|_{L^2}^2-\frac{\gamma_2(\gamma_3-2\gamma_2)}{16}\|v_0\|_{L^2}^2\right)\to\infty~~\text{as}~~|c|\to\infty.
\end{align*}
On the other hand, the Riemann-Lebesgue theorem implies that
\begin{align}\label{global:4}
    V\left(e^{\frac{i\gamma_1}{2}c\cdot x}u_0,e^{\frac{i\gamma_2}{2}c\cdot x}v_0,e^{i\frac{\gamma_3}{2}c\cdot x}w_0\right)=\Re\int e^{i\frac{\gamma_1+\gamma_2-\gamma_3}{2}c\cdot x}u_0v_0\bar{w}_0dx\to 0~~\text{as}~~|c|\to\infty.
\end{align}
This means that
\begin{align*}
    \frac{1}{2}K(u_0,v_0,w_0)-V\left(e^{\frac{i\gamma_1}{2}c\cdot x}u_0,e^{\frac{i\gamma_2}{2}c\cdot x}v_0,e^{i\frac{\gamma_3}{2}c\cdot x}w_0\right)
\end{align*}
is bounded as $|c|\to\infty$. Therefore, \eqref{global:3} holds for $|c|$ large enough, and so does \eqref{global:1}. In addition, by \eqref{global:4}, we obtain
\begin{align*}
    &N_{\frac{\gamma_3|c|^2}{8},c}\left(e^{\frac{i\gamma_1}{2}c\cdot x}u_0,e^{\frac{i\gamma_2}{2}c\cdot x}v_0,e^{i\frac{\gamma_3}{2}c\cdot x}w_0\right)\\
    =&K(u_0,v_0,w_0)+|c|^2\frac{\gamma_1(\gamma_3-2\gamma_1)}{8}\|u_0\|_{L^2}^2+|c|^2\frac{\gamma_2(\gamma_3-2\gamma_2)}{8}\|v_0\|_{L^2}^2\\
    &-3V\left(e^{\frac{i}{2}c\cdot x}u_0,e^{\frac{i}{2}c\cdot x}v_0,e^{i\frac{\gamma_3}{2}c\cdot x}w_0\right)\\
    \geq&K(u_0,v_0,w_0)+|c|^2\min\{\|u_0\|_{L^2}^2,\|v_0\|_{L^2}^2\}\frac{\min\{\gamma_1,\gamma_2\}(\gamma_3-\gamma_1-\gamma_2)}{4}\\
    &-3V\left(e^{\frac{i}{2}c\cdot x}u_0,e^{\frac{i}{2}c\cdot x}v_0,e^{i\frac{\gamma_3}{2}c\cdot x}w_0\right)\\
    \geq&0,
\end{align*}
for $|c|$ large enough since $\gamma_3>\gamma_1+\gamma_2$. Then, \eqref{global:2} holds. Hence, (i) holds.


To prove (ii). Let
\begin{align*}
    B_0=\frac{8}{\max\{\gamma_1,\gamma_3\}(3\gamma_2-\gamma_1-\gamma_3)}\mu_{\frac{\gamma_2}{4},\frac{c}{|c|}}.
\end{align*}
If the initial data $(u_0,v_0,w_0)\in H^1(\mathbb{R}^4)\times H^1(\mathbb{R}^4)\times H^1(\mathbb{R}^4)$ satisfies $\max\{\|u_0\|_{L^2}^2,\|w_0\|_{L^2}^2\}<B_0$, then we can obtain, for $|c|$ large enough,
\begin{align*}
    &S_{\frac{\gamma_2|c|^2}{4},c}\left(e^{\frac{i}{2}c\cdot x}u_0,e^{\frac{i}{2}c\cdot x}v_0,e^{i\frac{\gamma_3}{2}c\cdot x}w_0\right)<\mu_{\frac{|c|^2}{4},c},\\
    &N_{\frac{\gamma_2|c|^2}{4},c}\left(e^{\frac{i}{2}c\cdot x}u_0,e^{\frac{i}{2}c\cdot x}v_0,e^{i\frac{\gamma_3}{2}c\cdot x}w_0\right)\geq0.
\end{align*}
This means that, for $|c|$ large enough,
\begin{align*}
    \left(e^{\frac{i\gamma_1}{2}c\cdot x}u_0,e^{\frac{i\gamma_2}{2}c\cdot x}v_0,e^{i\frac{\gamma_3}{2}c\cdot x}w_0\right)\in A^+_{\omega,c}.
\end{align*}
Hence, (ii) holds.

Similarly, to prove (iii). Let
\begin{align*}
    C_0=\frac{8}{\max\{\gamma_2,\gamma_3\}(3\gamma_1-\gamma_2-\gamma_3)}\mu_{\frac{\gamma_1}{4},\frac{c}{|c|}}.
\end{align*}
If the initial data $(u_0,v_0,w_0)\in H^1(\mathbb{R}^4)\times H^1(\mathbb{R}^4)\times H^1(\mathbb{R}^4)$ satisfies $\max\{\|v_0\|_{L^2}^2,\|w_0\|_{L^2}^2\}<B_0$, then we can obtain, for $|c|$ large enough,
\begin{align*}
    &S_{\frac{\gamma_1|c|^2}{4},c}\left(e^{\frac{i}{2}c\cdot x}u_0,e^{\frac{i}{2}c\cdot x}v_0,e^{i\frac{\gamma_3}{2}c\cdot x}w_0\right)<\mu_{\frac{|c|^2}{4},c},\\
    &N_{\frac{\gamma_1|c|^2}{4},c}\left(e^{\frac{i}{2}c\cdot x}u_0,e^{\frac{i}{2}c\cdot x}v_0,e^{i\frac{\gamma_3}{2}c\cdot x}w_0\right)\geq0.
\end{align*}
This means that, for $|c|$ large enough,
\begin{align*}
    \left(e^{\frac{i\gamma_1}{2}c\cdot x}u_0,e^{\frac{i\gamma_2}{2}c\cdot x}v_0,e^{i\frac{\gamma_3}{2}c\cdot x}w_0\right)\in A^+_{\omega,c}.
\end{align*}
Hence, (iii) holds.

In particular, if $\gamma1+\gamma_2>\gamma_3$ and $\gamma_1=\gamma_2$,  there exists $D_0=\frac{8}{\gamma_3(2\gamma_1-\gamma_3)}\mu_{\frac{\gamma_1}{4},\frac{c}{|c|}}$ such that $\|w_0\|_{L^2}^2<D_0$. Then by (ii) and (iii), we can obtain the result.

Now we complete the proof of Theorem \ref{Thm3}.
\end{proof}

\appendix

\section{Appendix }

\subsection{Existence of the ground state}\label{section:ground}
In this section, we consider the existence of ground state of  the system \eqref{equ:s} with $c=0$. First, we give the following identities.

\begin{lemma}\label{lemma:pohozaev}
Let $(\phi,\psi,\varphi)$ be a solution of \eqref{equ:s} with $c=0$. Then we have the following Pohozaev identities
\begin{align*}
    K(\phi,\psi,\varphi)+\omega M(\phi,\psi,\varphi)=3\int \phi\psi\varphi,\\
     \frac{N-2}{2}K(\phi,\psi,\varphi)+\frac{N}{2}\omega M(\phi,\psi,\varphi)=N\int \phi\psi\varphi,
\end{align*}
where $K$ and $M$ are given by  \eqref{def:K} and \eqref{mass}, respectively.

\end{lemma}
\begin{proof}
By the similar argument as the single equation, we can obtain this two identities. Here we omit it.
\end{proof}

Next, we give the following lemma that is the existence of ground state.
\begin{lemma}
Let $N\leq 5$, and $\omega>0$. Then there exists a pair of positive radially symmetric functions $(\phi_0,\psi_0,\varphi_0)$ is a ground state for \eqref{equ:s} with $c=0$.
\end{lemma}
\begin{proof}
By the similar argument as \cite[Theorem 4.5]{HOT2013Poincare} or \cite[Section 4]{NP2021CCM}, we can obtain this result. Here we omit it.
\end{proof}

\subsection{Best constant in an inequality of Gagliardo-Nirenberg type}\label{section:GN}
In this subsection, we consider the best constant of the Gagliardo-Nirenberg type inequality in $\mathbb{R}^4$. By using the classical Gagliardo-Nirenberg inequality
\begin{align*}
    \|u\|_{L^3}\leq C\|\nabla u\|_{L^2}^{\frac{2}{3}}\|u\|_{L^2}^{\frac{1}{3}},
\end{align*}
we have
\begin{align}\label{GN:1}
    \int uvw\leq& C^3\|\nabla u\|_{L^2}^{\frac{2}{3}}\|u\|_{L^2}^{\frac{1}{3}}\|\nabla v\|_{L^2}^{\frac{2}{3}}\|v\|_{L^2}^{\frac{1}{3}}\|\nabla w\|_{L^2}^{\frac{2}{3}}\|w\|_{L^2}^{\frac{1}{3}}\notag\\
    \leq& C K(u,v,w)M(u,v,w)^{\frac{1}{2}},
\end{align}
where $K$ and $M$ are given by \eqref{def:K} and \eqref{mass}, respectively.

Now, we define
\begin{align*}
    \alpha=\inf\{J(u,v,w)~:~(u,v,w)\in A\},
\end{align*}
where
\begin{align*}
    &J(u,v,w)=\frac{K(u,v,w)M(u,v,w)^{\frac{1}{2}}}{\int uvw},\\
    &A=\left\{(u,v,w)\in H^1\times H^1\times H^1\backslash\{(0,0,0)\}~:~\int uvw>0\right\}.
\end{align*}
By \eqref{GN:1}, we obtain $\alpha>0$.
\begin{theorem}
Let $N=4$, then we have the following  Gagliardo-Nireberg inequality
\begin{align}\label{GN:sharp}
    \int uvw\leq C_{opt} K(u,v,w)M(u,v,w)^{\frac{1}{2}},
\end{align}
where
\begin{align*}
    C_{opt}=\frac{1}{2}\frac{1}{M(\phi,\psi,\varphi)^{\frac{1}{2}}}
\end{align*}
and $(\phi,\psi,\varphi)$ is the ground state of \eqref{equ:s} with $c=0$ and $\omega=1$.
In particular,  The inequality \eqref{GN:sharp} is sharp as the $"\leq"$ can be taken $"="$ by seting $(u,v,w)=(\phi,\psi,\varphi)$.



\end{theorem}
\begin{proof}
Let $(u_n,v_n,w_n)\subset A$ be a minimizing sequence for $J$. By the symmetric-decreasing rearrangement,  we may assume that $u_n, v_n, w_n$ are non-negative and radially symmetric functions in $H^1$. We define
\begin{align*}
    \Tilde{u}_n=a_nu_n(b_nx),~~\Tilde{v}_n=a_nv_n(b_nx),~~\Tilde{w}_n=a_nw_n(b_nx),
\end{align*}
where
\begin{align*}
    a_n=\frac{M(u_n,v_n,w_n)^{\frac{1}{2}}}{K(u_n,v_n,w_n)},~~b_n=\frac{M(u_n,v_n,w_n)^{\frac{1}{2}}}{K(u_n,v_n,w_n)^{\frac{1}{2}}},
\end{align*}
so that
\begin{align*}
    K(\Tilde{u}_n,\Tilde{v}_n,\Tilde{w}_n)=1,~~M(\Tilde{u}_n,\Tilde{v}_n,\Tilde{w}_n)=1
\end{align*}
and
\begin{align*}
    \frac{1}{\int \Tilde{u}_n\Tilde{v}_n\Tilde{w}_n}=J(\Tilde{u}_n,\Tilde{v}_n,\Tilde{w}_n)\to \alpha.
\end{align*}
Since $\{(\Tilde{u}_n,\Tilde{v}_n,\Tilde{w}_n)\}$ is bounded in $H^1\times H^1\times H^1$, there exists a subsequence still denote $\{(\Tilde{u}_n,\Tilde{v}_n,\Tilde{w}_n)\}$ such that
\begin{align*}
    \Tilde{u}_n\to u_0,~~\Tilde{v}_n\to v_0,~~\Tilde{w}_n\to w_0,~~\text{weakly in}~~H^1(\mathbb{R}^4)
\end{align*}
for some $(u_0,v_0,w_0)\in H^1\times H^1\times H^1 $. By Struss' compact embedding $H^1_r\subset L^3(\mathbb{R}^4)$
\begin{align*}
    \Tilde{u}_n\to u_0,~~\Tilde{v}_n\to v_0,~~\Tilde{w}_n\to w_0,~~\text{strongly in}~~L^3(\mathbb{R}^4).
\end{align*}
This yields
\begin{align*}
    \int u_0v_0w_0=\lim_{n\to\infty}\int \Tilde{u}_n\Tilde{v}_n\Tilde{w}_n=\frac{1}{\alpha}>0.
\end{align*}
On the other hand, by the lower semi-continuity of norms, we have
\begin{align*}
    K(u_0,v_0,w_0)\leq \lim_{n\to\infty} K(\Tilde{u}_n,\Tilde{v}_n,\Tilde{w}_n)=1,\\
    M(u_0,v_0,w_0)\leq \lim_{n\to\infty} M(\Tilde{u}_n,\Tilde{v}_n,\Tilde{w}_n)=1.
\end{align*}
Therefore, we obtain
\begin{align*}
    \alpha= J(u_0,v_0,w_0)=\frac{K(u_0,v_0,w_0)Q(u_0,v_0,w_0)^{\frac{1}{2}}}{\int u_0v_0w_0}\leq \lim_{n\to\infty}J(\Tilde{u}_n,\Tilde{v}_n,\Tilde{w}_n)=\alpha.
\end{align*}
We conclude that $(u_0,v_0,w_0)$ satisfies
\begin{align*}
    J(u_0,v_0,w_0)=\alpha, ~~K(u_0,v_0,w_0)=1,~~\int u_0,v_0,w_0=\frac{1}{\alpha},\\
    \Tilde{u}_n\to u_0,~\Tilde{v}_n\to v_0,~\Tilde{w}_n\to w_0~~~\text{strongly in}~~ H^1(\mathbb{R}^4).
\end{align*}
For any $(u,v,w)\in H^1\times H^1\times H^1$
\begin{align*}
    \frac{d}{ds}\Big|_{s=0}J(u_0+su,v_0+sv,w_0+sw)=0
\end{align*}
which implies that
\begin{align*}
    &\frac{M(u_0,v_0,w_0)^{\frac{1}{2}}}{\int u_0v_0w_0}\left(K^\prime(u_0,v_0,w_0)(u,v,w)+\frac{K(u_0,v_0,w_0)}{2M(u_0,v_0,w_0)}M^\prime(u_0,v_0,w_0)(u,v,w)\right)\\
    =&\frac{K(u_0,v_0,w_0)M(u_0,v_0,w_0)^{\frac{1}{2}}}{\left(\int u_0v_0w_0\right)^2}\left(\int u_0v_0w_0\right)^\prime(u,v,w).
\end{align*}
This yields
\begin{align*}
    K^\prime(u_0,v_0,w_0)(u,v,w)+\frac{1}{2}M^\prime(u_0,v_0,w_0)(u,v,w)=\alpha\left(\int u_0v_0w_0\right)^\prime(u,v,w).
\end{align*}
This equivalent to
\begin{align*}
    &2\int(\nabla u_0\cdot\nabla u+\nabla v_0\cdot\nabla v+\nabla w_0\cdot\nabla w)dx+\int (\gamma_1uu_0+\gamma_2vv_0+2\gamma_3ww_0)\\
    =&\alpha\int (uv_0w_0+u_0vw_0+u_0v_0w).
\end{align*}
We now define $(\phi,\psi,\varphi)=(\alpha u_0(\sqrt{2}x),\alpha v_0(\sqrt{2}x),\alpha w_0(\sqrt{2}x))$. Then $(\phi,\psi,\varphi)$ is a ground state solution of \eqref{equ:s} with $\omega=1$ and $c=0$. Since $(u_0,v_0,w_0)$ is a critical point of $J$, $(\phi,\psi,\varphi)$ is also a critical point of $J$. By the Pohozaev identity (see Lemma \ref{lemma:pohozaev}), we obtain
\begin{align*}
    K(\phi,\psi,\varphi)=2\int \phi\psi\varphi.
\end{align*}
Therefore
\begin{align*}
    J(\phi,\psi,\varphi)=2M(\phi,\psi,\varphi)^{\frac{1}{2}}.
\end{align*}
This complete the proof of this theorem.
\end{proof}




\renewcommand{\proofname}{\bf Proof.}

\noindent
{\bf Acknowledgments}

 Y.Li was supported by China Postdoctoral Science Foundation (No. 2021M701365) and the funding of innovating activities in Science and Technology of Hubei Province.


\vspace*{.5cm}


\bigskip


\begin{flushleft}
Yuan Li,\\
School of Mathematics and Statistics, Central China Normal University, Wuhan, PR China\\
E-mail: yli2021@ccnu.edu.cn
\end{flushleft}

\bigskip

\medskip

\end{document}